\documentclass{amsart}
\usepackage[nobysame]{amsrefs}
\usepackage{amssymb}
\usepackage[usesnames,svgnames]{xcolor}
\usepackage[sc]{mathpazo}
\usepackage{tikz,pgf}
  \pgfdeclarelayer{background}
  \pgfsetlayers{background,main}
\usepackage[colorlinks=true,
	urlcolor=MidnightBlue,
	citecolor=DarkGreen]{hyperref}
\usepackage{enumerate}
\usepackage{subfigure}
\newtheorem{theorem}{Theorem}[section]
\newtheorem{proposition}[theorem]{Proposition}

\newtheorem{lemma}[theorem]{Lemma}

\theoremstyle{remark}

\newtheorem{example}[theorem]{Example}
\newcommand{\alert}[1]{{\color{DarkGreen}\emph{#1}}}
\newcommand{\ie}{\text{i.e.}\;}
\newcommand{\LL}{\mathcal{L}}
\newcommand{\WW}{\mathcal{W}}
\newcommand{\camb}{\mathcal{C}_{\gamma}}

\newcommand{\cc}{\mathbf{c}}
\author{Henri M\"uhle}
\address{LIAFA, Universit{\'e} Paris Diderot, Case 7014, F-75205 Paris Cedex 13, France}
\email{henri.muehle@liafa.univ-paris-diderot.fr}
\keywords{Cambrian semilattice, Tamari lattice, Coxeter group, sortable elements, trimness}
\subjclass[2010]{05E15 (Primary), and 06D75, 20F55 (Secondary)}
\thanks{This work was funded by the FWF Research Grant No. Z130-N13, and by a Public Grant overseen by the French National Research Agency (ANR) as part of the ``Investissements d'Avenir'' Program (Reference: ANR-10-LABX-0098).}
\title{Trimness of Closed Intervals in Cambrian Semilattices}
\begin{document}

\begin{abstract}
	In this article, we give a short algebraic proof that all closed intervals in a $\gamma$-Cambrian semilattice $\mathcal{C}_{\gamma}$ are trim for any Coxeter group $W$ and any Coxeter element $\gamma\in W$.  This means that if such an interval has length $k$, then there exists a maximal chain of length $k$ consisting of left-modular elements, and there are precisely $k$ join- and $k$ meet-irreducible elements in this interval.  Consequently every graded interval in $\mathcal{C}_{\gamma}$ is distributive.  This problem was open for any Coxeter group that is not a Weyl group. 

%
\end{abstract}

\maketitle

\section{Introduction}
  \label{sec:introduction}
The $\gamma$-Cambrian semilattice, denoted by $\camb$ and parametrized by a Coxeter group $W$ and a Coxeter element $\gamma\in W$, was introduced by N.~Reading and D.~Speyer in~\cite{reading11sortable}, generalizing N.~Reading's construction for finite $W$~\cites{reading06cambrian,reading07sortable}.  This family of semilattices can be seen as a generalization of the famous Tamari lattices to all Coxeter groups, in the sense that the $\gamma$-Cambrian semilattice associated with the symmetric group $\mathfrak{S}_{n}$ and the Coxeter element $\gamma=(1\;2\;\ldots\;n)$ is isomorphic to the Tamari lattice of parameter $n$.  The Tamari lattices play an important role in algebraic and geometric combinatorics, and frequently occur in many seemingly unrelated branches of mathematics~\cite{hoissen12associahedra}.  

Previously, many of the properties of the Tamari lattices have been generalized to the $\gamma$-Cambrian semilattices, such as EL-shellability and the topology of their order complexes~\cites{kallipoliti13on,pilaud13el}, congruence-uniformity~\cites{reading03lattice,reading11sortable,santocanale13sublattices}, or semidistributivity~\cite{reading11sortable}.  Another property that the Tamari lattices enjoy is trimness, introduced by H.~Thomas~\cite{thomas06analogue}.  A finite lattice of length $k$ is \alert{trim} if it possesses a left-modular chain of length $k$, and if it has precisely $k$ join- and $k$ meet-irreducible elements.  An interesting observation that goes back to G.~Markowsky~\cite{markowsky92primes} is the fact that graded trim lattices are distributive.  Hence trimness can be seen as a generalization of distributivity to ungraded lattices.  It was conjectured in \cite{thomas06analogue} that $\camb$ is trim for any finite Coxeter group $W$, and any Coxeter element $\gamma\in W$.  In the same paper, it was shown that this conjecture is true when $W$ is of type $A$ or $B$, using the permutation representation of these groups and the definition of $\camb$ as posets on certain pattern-avoiding permutations~\cite{thomas06analogue}*{Theorems~8~and~9}.  Subsequently, it was shown in \cite{ingalls09noncrossing}*{Theorem~4.17} that said conjecture holds when $W$ is a Weyl group, using the definition of $\camb$ as posets on torsion classes of certain quiver representations.  It is the purpose of this paper to prove the trimness of $\camb$ in full generality, namely when $W$ is an arbitrary (perhaps infinite) Coxeter group, and $\gamma\in W$ is any Coxeter element.  More precisely, we prove the following theorem.

\begin{theorem}\label{thm:cambrian_trim}
	Every closed interval in a $\gamma$-Cambrian semilattice $\camb$ is trim for any Coxeter group $W$ and any Coxeter element $\gamma\in W$.  In particular, every graded closed interval in $\camb$ is distributive. 
\end{theorem}

Recall that when $W$ is infinite, $\camb$ is only a semilattice since it does not possess a maximal element.  However, any closed interval in $\camb$ is a finite lattice in its own right, and Theorem~\ref{thm:cambrian_trim} thus implies the local trimness of $\camb$. 

The proof of Theorem~\ref{thm:cambrian_trim} uses the definition of $\camb$ in terms of sortable elements, and is thus uniform.  One ingredient for proving Theorem~\ref{thm:cambrian_trim} is the semidistributivity of $\camb$ that was established in \cite{reading11sortable}*{Section~8}.  Semidistributivity can be seen as another generalization of distributivity to ungraded lattices, but it is different from trimness. Consider for instance the lattice in Figure~\ref{fig:trim_not_semi}.  This is a trim lattice, since it has length $4$, it has exactly four join- and meet-irreducible elements, and the highlighted chain consists of left-modular elements.  However, it is not semidistributive, since it is precisely one of the minimal lattices that were used in \cite{davey75characterization} to characterize the obstructions to semidistributivity.  Conversely, Figure~\ref{fig:semi_not_trim} shows a semidistributive lattice that is not trim, since it has four join- and four meet-irreducible elements, but length only three.

\begin{figure}
	\centering
	\subfigure[A trim lattice that is not semidistributive.]{\label{fig:trim_not_semi}\begin{tikzpicture}\small
		\def\x{.75};
		\def\y{.75};
		\draw(3*\x,1*\y) node[fill,circle,scale=.5](n1){};
		\draw(2*\x,2*\y) node[fill,circle,scale=.5](n2){};
		\draw(1*\x,3*\y) node[fill,circle,scale=.5](n3){};
		\draw(3*\x,3*\y) node[fill,circle,scale=.5](n4){};
		\draw(5*\x,3*\y) node[fill,circle,scale=.5](n5){};
		\draw(4*\x,4*\y) node[fill,circle,scale=.5](n6){};
		\draw(3*\x,5*\y) node[fill,circle,scale=.5](n7){};
		\draw[very thick](n1) -- (n2);
		\draw(n1) -- (n5);
		\draw(n2) -- (n3);
		\draw[very thick](n2) -- (n4);
		\draw(n3) -- (n7);
		\draw[very thick](n4) -- (n6);
		\draw(n5) -- (n6);
		\draw[very thick](n6) -- (n7);
	\end{tikzpicture}}
	\hspace*{1cm}
	\subfigure[A semidistributive lattice that is not trim.]{\label{fig:semi_not_trim}\begin{tikzpicture}\small
		\def\x{1};
		\def\y{1};
		\draw(0*\x,2*\y) node{};
		\draw(4*\x,2*\y) node{};
		\draw(2*\x,1*\y) node[fill,circle,scale=.5](n1){};
		\draw(1*\x,2*\y) node[fill,circle,scale=.5](n2){};
		\draw(3*\x,2*\y) node[fill,circle,scale=.5](n3){};
		\draw(1*\x,3*\y) node[fill,circle,scale=.5](n4){};
		\draw(3*\x,3*\y) node[fill,circle,scale=.5](n5){};
		\draw(2*\x,4*\y) node[fill,circle,scale=.5](n6){};
		\draw(n1) -- (n2);
		\draw(n1) -- (n3);
		\draw(n2) -- (n4);
		\draw(n3) -- (n5);
		\draw(n4) -- (n6);
		\draw(n5) -- (n6);
	\end{tikzpicture}}
	\caption{Different generalizations of distributivity.}
	\label{fig:semi_trim}
\end{figure}
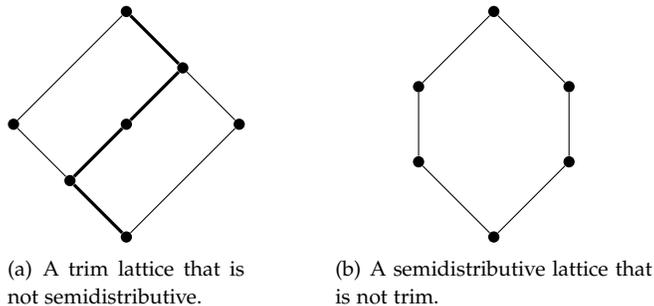

\smallskip

We recall the necessary background on Coxeter groups in Section~\ref{sec:coxeter_groups}, and recall the definition of the $\gamma$-Cambrian semilattices in Section~\ref{sec:cambrian_lattices}.  We define the non-standard poset-theoretical concepts whenever needed, and refer to \cite{davey02introduction} for any undefined terminology.  In Section~\ref{sec:trimness}, we prove Theorem~\ref{thm:cambrian_trim}. 

\section{Preliminaries}
  \label{sec:preliminaries}
In this section we give the definitions needed in the remainder of this article.  For more information on Coxeter groups, we refer to \cite{bjorner05combinatorics} and \cite{humphreys90reflection}.  An excellent exposition on $\gamma$-Cambrian semilattices is \cite{reading11sortable}.  Throughout the article we use the abbreviation $[n]=\{1,2,\ldots,n\}$.

\subsection{Coxeter Groups}
  \label{sec:coxeter_groups}
A \alert{Coxeter system} $(W,S)$ is a pair given by a group $W$ (whose identity is denoted by $\varepsilon$) and a subset $S\subseteq W$ with $S=\{s_{1},s_{2},\ldots,s_{n}\}$ such that $W$ admits the presentation 
\begin{displaymath}
	W=\bigl\langle S\mid (s_{i}s_{j})^{m_{i,j}}=\varepsilon\;\text{for}\;i,j\in[n]\bigr\rangle.
\end{displaymath}
In this presentation each $m_{i,j}$ is either a positive integer or the formal symbol $\infty$, and these numbers satisfy $m_{i,j}=m_{j,i}\geq 1$ for $i,j\in[n]$, and $m_{i,j}=1$ if and only if $i=j$.  In particular, $m_{i,j}=\infty$ means that there is no relation between the generators $s_{i}$ and $s_{j}$.  The elements $s_{1},s_{2},\ldots,s_{n}$ are called the \alert{simple generators} of $W$ and $n$ is called the \alert{rank} of $W$.  Since $S$ generates the group $W$, every $w\in W$ can be written as a product of those simple generators.  This gives rise to the length function
\begin{displaymath}
	\ell_{S}(w)=\min\bigl\{k\mid w=s_{i_{1}}s_{i_{2}}\cdots s_{i_{k}},\;\text{where}\;s_{i_{j}}\in S\;\text{for}\;j\in[k]\bigr\}.
\end{displaymath}
Now we can define the \alert{(right) weak order} on $W$ by
\begin{align}\label{eq:weak_order}
	u\leq_{S}v\quad\text{if and only if}\quad\ell_{S}(v)=\ell_{S}(u)+\ell_{S}(u^{-1}v).
\end{align}
The poset $\WW=(W,\leq_{S})$ forms a graded meet-semilattice with rank function $\ell_{S}(w)$.  If $W$ is finite, then there exists a unique longest element $w_{o}\in W$, and $\WW$ is thus a lattice.  In general $\WW$ is \alert{finitary}, meaning that every principal order ideal of $\WW$ is finite. 

For any $J\subseteq S$ the \alert{standard parabolic subgroup} $W_{J}$ is the subgroup of $W$ generated by $J$.  For $s\in S$ we frequently use the abbreviation $\langle s\rangle=S\setminus\{s\}$, and thus $W_{\langle s\rangle}=W_{S\setminus\{s\}}$.  

\subsection{Cambrian Lattices}
  \label{sec:cambrian_lattices}
For $w\in W$ with $\ell_{S}(w)=k$, an expression $w=s_{i_{1}}s_{i_{2}}\cdots s_{i_{k}}$ is called a \alert{reduced word} for $w$.  A \alert{Coxeter element} is an element $\gamma\in W$ that has a reduced word which is a permutation of the simple generators.  Without loss of generality, we can restrict our attention to $\gamma=s_{1}s_{2}\cdots s_{n}$, and define the half-infinite word
\begin{displaymath}
	\gamma^{\infty} = s_{1}s_{2}\cdots s_{n}\vert s_{1}s_{2}\cdots s_{n}\vert s_{1}\cdots.
\end{displaymath}
The vertical bars serve only as ``dividers'' and have no influence on the structure of $\gamma^{\infty}$.  However, they yield an intuitive notion of a \alert{block} of $\gamma^{\infty}$.  It is easy to see that every reduced word for $w\in W$ can be regarded as a subword of $\gamma^{\infty}$.  Among all reduced words for $w$ there is a unique reduced word which is lexicographically first as a subword of $\gamma^{\infty}$.  We refer to this as the \alert{$\gamma$-sorting word} of $w$.  We say that $w$ is \alert{$\gamma$-sortable} if the blocks of the $\gamma$-sorting word of $w$ form a weakly decreasing sequence with respect to inclusion.  Let $C_{\gamma}$ denote the set of all $\gamma$-sortable elements of $W$. 

\begin{example}\label{ex:cambrian_1}
	Let $\tilde{C}_{3}$ be the affine Coxeter group generated by $\{s_{0},s_{1},s_{2},s_{3}\}$ with $(s_{0}s_{1})^{4}=(s_{0}s_{2})^{2}=(s_{0}s_{3})^{2}=(s_{1}s_{2})^{3}=(s_{1}s_{3})^{2}=(s_{2}s_{3})^{4}=\varepsilon$, and consider the Coxeter element 	$\gamma=s_{0}s_{1}s_{2}s_{3}$.  The element $w=s_{0}s_{2}s_{3}s_{2}\in\tilde{C}_{3}$ has precisely four reduced words, namely
	\begin{displaymath}\begin{aligned}
		& \mathbf{s_{0}s_{2}s_{3}\vert s_{2}}, && s_{2}\vert s_{0}s_{3}\vert s_{2}, && s_{2}s_{3}\vert s_{0}s_{2}, && s_{2}s_{3}\vert s_{2}\vert s_{0},
	\end{aligned}\end{displaymath}
	and the highlighted one is the $\gamma$-sorting word of $w$.  Its blocks are $b_{1}=\{s_{0},s_{2},s_{3}\}$ and $b_{2}=\{s_{2}\}$, and we have $b_{1}\supseteq b_{2}$.  Hence $w$ is $\gamma$-sortable.  On the other hand consider the element $w'=s_{0}s_{2}s_{3}s_{1}\in\tilde{C}_{3}$.  It has precisely five reduced words, namely
	\begin{displaymath}\begin{aligned}
		& \mathbf{s_{0}s_{2}s_{3}\vert s_{1}}, && s_{0}s_{2}\vert s_{1}s_{3}, && s_{2}\vert s_{0}s_{1}s_{3}, 
		&& s_{2}\vert s_{0}s_{3}\vert s_{1}, && s_{2}s_{3}\vert s_{0}s_{1},
	\end{aligned}\end{displaymath}
	and again the highlighted one is the $\gamma$-sorting word of $w'$.  Its blocks are $b'_{1}=\{s_{0},s_{2},s_{3}\}$ and $b'_{2}=\{s_{1}\}$, and we have $b'_{1}\not\supseteq b'_{2}$.  Hence $w'$ is not $\gamma$-sortable.
\end{example}

A simple generator $s\in S$ is \alert{initial in $\gamma$} if there exists some reduced word for $\gamma$ which starts with $s$.  It follows immediately from \eqref{eq:weak_order} that $s\in S$ is initial in $\gamma$ if and only if $s\leq_{S}\gamma$.  The $\gamma$-sortable elements of $W$ admit a nice recursive characterization as the next proposition shows.

\begin{proposition}[\cite{reading11sortable}*{Proposition~2.29}]\label{prop:sortable_recursion}
	Let $(W,S)$ be a Coxeter system, let $\gamma\in W$ be a Coxeter element, and let $s\in S$ be initial in $\gamma$.  An element $w\in W$ is $\gamma$-sortable if and only if either
	\begin{enumerate}[(i)]
		\item $s\leq_{S}w$ and $sw$ is $s\gamma s$-sortable, \quad or
		\item $s\not\leq_{S}w$ and w is an $s\gamma$-sortable element of $W_{\langle s\rangle}$.
	\end{enumerate}
\end{proposition}

It is rather straightforward from the definition that the set $C_{\gamma}$ of $\gamma$-sortable elements does not depend on the reduced word for $\gamma$, see for instance \cite{reading11sortable}*{Section~2.7} for a detailed explanation.  A nice consequence of Proposition~\ref{prop:sortable_recursion} is that we can use induction on rank and length when proving statements about $\gamma$-sortable elements.  More precisely, Proposition~\ref{prop:sortable_recursion} states that any $\gamma$-sortable element $w\in W$ falls into one of the two categories with respect to an initial generator $s$ of $\gamma$: either $s\leq_{S}w$ or $s\not\leq_{S}w$.  In the first case we can reduce the length of $w$ and leave the rank of $W$ fixed, in the second case we can leave the length of $w$ fixed, and consider it as an element in a standard parabolic subgroup of $W$ of smaller rank.  See also the paragraph after \cite{reading11sortable}*{Proposition~2.29}.

The next result indicates the special role of the $\gamma$-sortable elements among the elements of $W$.

\begin{theorem}[\cite{reading11sortable}*{Theorem~7.1}]\label{thm:sortable_meets_joins}
	Let $A$ be a collection of $\gamma$-sortable elements of $W$.  If $A$ is nonempty, then $\bigwedge A$ is $\gamma$-sortable. If $A$ has an upper bound, then $\bigvee A$ is $\gamma$-sortable. 
\end{theorem}

This theorem implies that $C_{\gamma}$ equipped with the weak order constitutes a sub-semilattice of $\WW$, the so-called 
\alert{$\gamma$-Cambrian semilattice} of $W$, usually denoted by $\camb=(C_{\gamma},\leq_{\gamma})$, where $\leq_{\gamma}$ is the restriction of $\leq_{S}$ to $C_{\gamma}$.  Moreover, for two $\gamma$-sortable elements $u,v\in C_{\gamma}$ with $u\leq_{\gamma}v$, we denote the induced interval in $\camb$ by $[u,v]_{\gamma}=\{x\in C_{\gamma}\mid u\leq_{\gamma}x\leq_{\gamma}v\}$.  Figure~\ref{fig:cambrian_c3_interval} shows an interval in the $\gamma$-Cambrian semilattice of the affine Coxeter group $\tilde{C}_{3}$ and the Coxeter element $\gamma$ from Example~\ref{ex:cambrian_1}.

\begin{figure}
	\centering
	\begin{tikzpicture}\small
		\def\x{1.12};
		\def\y{.85};
		\draw(4.5*\x,1*\y) node(n1){$\varepsilon$};
		\draw(3*\x,2.5*\y) node(n2){$s_{0}$};
		\draw(4.5*\x,3*\y) node(n3){$s_{2}$};
		\draw(1*\x,4.5*\y) node(n4){$s_{0}s_{1}$};
		\draw(6.25*\x,4.75*\y) node(n5){$s_{2}s_{3}$};
		\draw(1*\x,5.5*\y) node(n6){$s_{0}s_{1}s_{2}$};
		\draw(3*\x,5.5*\y) node(n7){$s_{0}s_{2}$};
		\draw(1*\x,6.5*\y) node(n8){\hspace*{-.25cm}$s_{0}s_{1}s_{2}s_{1}$};
		\draw(2*\x,6.5*\y) node(n9){\hspace*{.25cm}$s_{0}s_{1}s_{2}s_{3}$};
		\draw(4.75*\x,7.25*\y) node(n10){$s_{0}s_{2}s_{3}$};
		\draw(2*\x,7.5*\y) node(n11){$s_{0}s_{1}s_{2}s_{3}s_{1}$};
		\draw(9*\x,7.5*\y) node(n12){$s_{2}s_{3}s_{2}$};
		\draw(11.5*\x,8*\y) node(n13){$s_{3}$};
		\draw(2.75*\x,8.25*\y) node(n14){$s_{0}s_{1}s_{2}s_{3}s_{1}s_{2}$};
		\draw(4.5*\x,9*\y) node(n15){$s_{0}s_{1}s_{2}s_{3}s_{2}$};
		\draw(2.75*\x,9.25*\y) node(n16){$s_{0}s_{1}s_{2}s_{3}s_{1}s_{2}s_{3}$};
		\draw(10*\x,9.5*\y) node(n17){$s_{0}s_{3}$};
		\draw(4.5*\x,10*\y) node(n18){$s_{0}s_{1}s_{2}s_{3}s_{1}s_{2}s_{1}$};
		\draw(7.5*\x,10*\y) node(n19){$s_{0}s_{2}s_{3}s_{2}$};
		\draw(11.5*\x,10*\y) node(n20){$s_{2}s_{3}s_{2}s_{3}$};
		\draw(4.5*\x,11*\y) node(n21){$s_{0}s_{1}s_{2}s_{3}s_{1}s_{2}s_{3}s_{1}$};
		\draw(8*\x,11.5*\y) node(n22){$s_{0}s_{1}s_{3}$};
		\draw(5.5*\x,12*\y) node(n23){$s_{0}s_{1}s_{2}s_{3}s_{1}s_{2}s_{3}s_{1}s_{2}$};
		\draw(10*\x,12.5*\y) node(n24){$s_{0}s_{2}s_{3}s_{2}s_{3}$};
		\draw(8*\x,12.5*\y) node(n25){$s_{0}s_{1}s_{2}s_{3}s_{2}s_{3}$};
		\draw(8*\x,14.5*\y) node(n26){$s_{0}s_{1}s_{2}s_{3}s_{1}s_{2}s_{3}s_{1}s_{2}s_{3}$};
		\draw[very thick](n1) -- (n2);
		\draw(n1) -- (n3);
		\draw(n1) -- (n13);
		\draw[very thick](n2) -- (n4);
		\draw(n2) -- (n7);
		\draw(n2) -- (n17);
		\draw(n3) -- (n7);
		\draw(n3) -- (n5);
		\draw[very thick](n4) -- (n6);
		\draw(n4) -- (n22);
		\draw(n5) -- (n10);
		\draw(n5) -- (n12);
		\draw(n6) -- (n8);
		\draw[very thick](n6) -- (n9);
		\draw(n7) -- (n8);
		\draw(n7) -- (n10);
		\draw(n8) -- (n11);
		\draw[very thick](n9) -- (n11);
		\draw(n9) -- (n15);
		\draw(n10) -- (n16);
		\draw(n10) -- (n19);
		\draw[very thick](n11) -- (n14);
		\draw(n12) -- (n19);
		\draw(n12) -- (n20);
		\draw(n13) -- (n17);
		\draw(n13) -- (n20);
		\draw[very thick](n14) -- (n16);
		\draw(n14) -- (n18);
		\draw(n15) -- (n18);
		\draw(n15) -- (n25);
		\draw[very thick](n16) -- (n21);
		\draw(n17) -- (n22);
		\draw(n17) -- (n24);
		\draw(n18) -- (n21);
		\draw(n19) -- (n23);
		\draw(n19) -- (n24);
		\draw(n20) -- (n24);
		\draw[very thick](n21) -- (n23);
		\draw(n22) -- (n25);
		\draw[very thick](n23) -- (n26);
		\draw(n24) -- (n26);
		\draw(n25) -- (n26);
	\end{tikzpicture}
	\caption{An interval in the $s_{0}s_{1}s_{2}s_{3}$-Cambrian semilattice of $\tilde{C}_{3}$.}
	\label{fig:cambrian_c3_interval}
\end{figure}
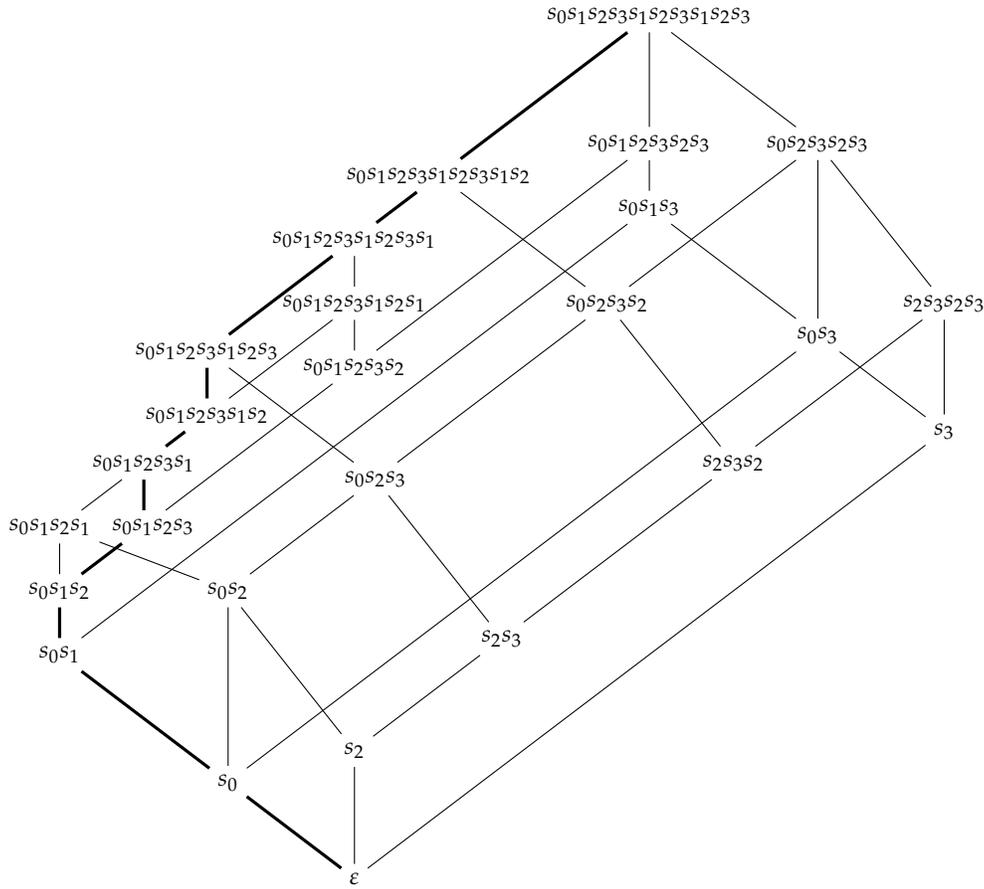

As observed in \cite{reading11sortable}, the $\gamma$-Cambrian semilattices have some important lattice-theoretic properties.  Recall that a lattice $\LL=(L,\leq)$ is \alert{join-semidistributive} if it satisfies
\begin{equation}
	\label{eq:join_semi}x\vee y=x\vee z\quad\text{implies}\quad x\vee y=x\vee(y\wedge z)
\end{equation}
for all $x,y,z\in L$.  It is \alert{meet-semidistributive} if its dual is join-semidistributive, and \alert{semidistributive} if it is both join- and meet-semidistributive.  

An important property of finite join-semidistributive lattices is that these are precisely the finite lattices in which each element has a so-called \alert{canonical join-representation}, \ie each element can be written as the join over a minimal, irredundant set of elements in the lattice.  (Here, minimal has a double meaning: on the one hand, the set in question should have the smallest possible cardinality, and on the other hand, its elements should be as close to the bottom of the lattice as possible.)  A reference for this result is for instance \cite{freese95free}*{Theorem~2.24}.  We also refer the reader to \cite{freese95free}*{Chapters~I.3~and~II.5} for more background on canonical forms in lattices and semidistributivity.  By duality, we establish the same connection between canonical meet-representations and meet-semidistributivity.  Now, \cite{reading11sortable}*{Theorem~8.1} states that any element in a Coxeter group $W$ has a canonical join-representation in the weak order semilattice, and  \cite{reading11sortable}*{Proposition~8.2} states that for any Coxeter element $\gamma\in W$ and any $\gamma$-sortable element $w\in W$, the elements in the canonical join-representation of $w$ (with respect to the weak order semilattice) are again $\gamma$-sortable, and in particular form a canonical join-representation in the corresponding $\gamma$-Cambrian semilattice.  As a consequence, any closed interval in the $\gamma$-Cambrian semilattice $\camb$ of $W$ is join-semidistributive.  Proposition~2.19 in \cite{reading11sortable} implies that the dual of any closed interval in $\camb$ appears as some interval on some Cambrian semilattice of $W$, and we thus have the following result.

\begin{proposition}\label{prop:cambrian_semidistributive}
	Every closed interval in $\camb$ is semidistributive for any Coxeter group $W$ and any Coxeter element $\gamma\in W$. 
\end{proposition}

\section{Trimness of Closed Intervals of $\camb$}
  \label{sec:trimness}
Let $\LL=(L,\leq)$ be a lattice with least element $\hat{0}$ and greatest element $\hat{1}$.  For two elements $x,y\in L$ with $x<y$ we say that $x$ \alert{is a lower cover of} $y$ if there is no element $z\in L$ with $x<z<y$, and we usually denote this by $x\lessdot y$.  In this situation we also say that $y$ \alert{is an upper cover of} $x$, or that $x$ and $y$ form a \alert{cover relation}.  Recall that $x\in L\setminus\{\hat{0}\}$ is \alert{join-irreducible} if for every set $X\subseteq L$ with $x=\bigvee X$ we have $x\in X$.  In other words $x$ cannot be expressed as the join of some elements of $\LL$ strictly below $x$.  Join-irreducible elements can be easily spotted in the Hasse diagram of $\LL$ since they are precisely the elements with a unique lower cover.  Let $\mathcal{J}(\LL)$ denote the set of join-irreducible elements of $\LL$.  Dually, $x\in L\setminus\{\hat{1}\}$ is \alert{meet-irreducible} if it cannot be expressed as the meet of some elements of $\LL$ strictly above $x$.  The meet-irreducible elements of $\LL$ are precisely the elements with exactly one upper cover, and we denote the set of all meet-irreducible elements of $\LL$ by $\mathcal{M}(\LL)$.  

A \alert{maximal chain} of $\LL$ is a sequence of cover relations $\hat{0}\lessdot x_{1}\lessdot x_{2}\lessdot\cdots\lessdot x_{s}\lessdot\hat{1}$.  The maximum length of a maximal chain of $\LL$ is the \alert{length} of $\LL$, usually denoted by $\ell(\LL)$.  It is easy to see that the number of join- and meet-irreducibles, respectively, cannot be less than the length of a lattice.  In that sense $\LL$ is \alert{extremal} if $\bigl\lvert\mathcal{J}(\LL)\bigr\rvert=\ell(\LL)=\bigl\lvert\mathcal{M}(\LL)\bigr\rvert$~\cite{markowsky92primes}.   

\begin{example}\label{ex:cambrian_2}
	Let us continue with Example~\ref{ex:cambrian_1}, and consider the lattice in Figure~\ref{fig:cambrian_c3_interval}.  The ten join-irreducible elements of this lattice are
	\begin{displaymath}\begin{aligned}
		& s_{0}, && s_{2}, && s_{3}, && s_{0}s_{1}, && s_{2}s_{3}, && s_{0}s_{1}s_{2}, && s_{2}s_{3}s_{2}, && s_{0}s_{1}s_{2}s_{3}, && s_{0}s_{1}s_{2}s_{3}s_{2}, && s_{0}s_{1}s_{2}s_{3}s_{1}s_{2},
	\end{aligned}\end{displaymath}
	and the ten meet-irreducible elements are
	\begin{displaymath}\begin{aligned}
		& s_{0}s_{1}s_{3}, && s_{0}s_{1}s_{2}s_{1}, && s_{2}s_{3}s_{2}s_{3}, && s_{0}s_{1}s_{2}s_{3}s_{1},\\
		& s_{0}s_{2}s_{3}s_{2}s_{3}, && s_{0}s_{1}s_{2}s_{3}s_{2}s_{3}, && s_{0}s_{1}s_{2}s_{3}s_{1}s_{2}s_{1},\\ 
		& s_{0}s_{1}s_{2}s_{3}s_{1}s_{2}s_{3}, && s_{0}s_{1}s_{2}s_{3}s_{1}s_{2}s_{3}s_{1}, && s_{0}s_{1}s_{2}s_{3}s_{1}s_{2}s_{3}s_{1}s_{2}.
	\end{aligned}\end{displaymath}
	Since for instance the chain on the left boundary of the Hasse diagram of this lattice consists of eleven elements, and we can quickly check that there exists no chain with more elements, we conclude that this lattice has length ten, and is thus extremal.  
\end{example}

The following statement shows that extremality can be seen as a generalization of distributivity to ungraded lattices. 

\begin{theorem}[\cite{markowsky92primes}*{Theorem~17}]\label{thm:graded_extremal_distributive}
	A graded extremal lattice is distributive.
\end{theorem}

However, G.~Markowsky observed in \cite{markowsky92primes}*{Theorem~14(ii)} that any finite lattice can be embedded as an interval into an extremal lattice, which in particular means that extremality is not inherited to intervals.  In order to overcome this issue, H.~Thomas introduced a refinement of extremality that we will describe next.  Recall that $x\in L$ is \alert{left-modular} if for all $y,z\in L$ with $y<z$ we have 
\begin{equation}\label{eq:left_modular}
	(y\vee x)\wedge z = y\vee(x\wedge z).
\end{equation}
If $\LL$ has length $n$ and there exists a maximal chain $\hat{0}=x_{0}\lessdot x_{1}\lessdot\cdots\lessdot x_{n}=\hat{1}$ such that $x_{i}$ is left-modular for all $i\in\{0,1,\ldots,n\}$, then $\LL$ is \alert{left-modular}.  Left-modularity admits the following handy characterization.

\begin{theorem}[\cite{liu00left}*{Theorem~1.4}]\label{thm:left_modularity}
	Let $\LL=(L,\leq)$ be a finite lattice and let $x\in L$.  The following are equivalent:
	\begin{enumerate}[(i)]
		\item the element $x$ is left-modular;
		\item for any $y,z\in L$ with $y<z$, we have $x\wedge y\neq x\wedge z$ or $x\vee y\neq x\vee z$;\quad and
		\item for any $y,z\in L$ with $y\lessdot z$, we have $x\wedge y=x\wedge z$ or $x\vee y=x\vee z$ but not both.
	\end{enumerate}
\end{theorem}

A lattice that is both extremal and left-modular is \alert{trim}~\cite{thomas06analogue}, and this property is inherited to intervals.

\begin{theorem}[\cite{thomas06analogue}*{Theorem~1}]\label{thm:trim_intervals}
	Any interval of a trim lattice is trim.
\end{theorem}

Let $\gamma=s_{1}s_{2}\cdots s_{n}$, let $w\in C_{\gamma}$ such that $\ell_{S}(w)=k$, and let $w=r_{1}r_{2}\cdots r_{k}$ be the $\gamma$-sorting word of $w$.  It is immediate that the length of the interval $[\varepsilon,w]_{\gamma}$ is $k$ and that the chain
\begin{equation}\label{eq:chain}
	\cc:\varepsilon=x_{0}\lessdot_{\gamma}x_{1}\lessdot_{\gamma}\cdots\lessdot_{\gamma}x_{k}=w,
\end{equation}
with $x_{i}=r_{1}r_{2}\cdots r_{i}$ for $i\in[k]$ is a maximal chain, see also \cite{kallipoliti13on}*{Remark~3.12}.  We prove next that this chain is in fact a left-modular chain of $[\varepsilon,w]_{\gamma}$.  For this we need the following result.

\begin{lemma}[\cite{kallipoliti13on}*{Lemma~3.10}]\label{lem:covering_join}
	Let $\gamma=s_{1}s_{2}\cdots s_{n}$, and let $u,v\in C_{\gamma}$ with $u\leq_{\gamma}v$.  If $s_{1}\not\leq_{\gamma}u$ and $s_{1}\leq_{\gamma}v$, then the join $s_{1}\vee u$ covers $u$ in $\camb$.  In particular if $u\lessdot_{\gamma}v$, then $s_{1}\vee u=v$.
\end{lemma}

It is the purpose of this article to establish the trimness of closed intervals in $\camb$, and as a first step we prove the following proposition, which asserts that the chain in \eqref{eq:chain} is always left-modular.  

\begin{proposition}\label{prop:cambrian_left_modular}
	For $w\in C_{\gamma}$, the chain in \eqref{eq:chain} is a left-modular maximal chain in $[\varepsilon,w]_{\gamma}$.
\end{proposition}

The proof of Proposition~\ref{prop:cambrian_left_modular} is split into four parts, which are treated in Lemmas~\ref{lem:clm_1}--\ref{lem:clm_4} below, and we proceed by induction on rank and length.  Suppose that $W$ has rank $n$ and that $w\in W$ is $\gamma$-sortable with $\ell_{S}(w)=k$.  In view of Theorem~\ref{thm:left_modularity}, proving Proposition~\ref{prop:cambrian_left_modular} amounts to showing that for any $y\lessdot_{\gamma}z\leq_{\gamma}w$ we have 
\begin{equation}\label{eq:target}
	\text{either}\quad x_{j}\wedge y=x_{j}\wedge z\quad\text{or}\quad x_{j}\vee y=x_{j}\vee z\quad\text{(but not both)},
\end{equation}
for any $j\in[k]$, since the identity is trivially left-modular.  The base cases for our induction on rank and length---namely when $n=2$ or $k=2$---are trivial.  Throughout the next four lemmas, we assume that the claim of Proposition~\ref{prop:cambrian_left_modular} is true for all standard parabolic subgroups of $W$ having rank $<n$, and for all $\gamma'$-sortable elements $w'\in W$ for some Coxeter element $\gamma'\in W$ with $\ell_{S}(w')<k$.  

\begin{lemma}\label{lem:clm_1}
	If $s_{1}\leq_{\gamma}y\lessdot_{\gamma}z\leq_{\gamma}w$, then \eqref{eq:target} is satisfied.
\end{lemma}
\begin{proof}
	Proposition~2.18 in \cite{reading11sortable} implies that the map $x\mapsto s_{1}x$ is an isomorphism from the weak order interval $[s_{1},w]$ to the weak order interval $[\varepsilon,s_{1}w]$, and Proposition~\ref{prop:sortable_recursion} implies that this isomorphism restricts to an isomorphism from the closed interval $[\varepsilon,w]_{\gamma}$ in $\camb$ to the interval $[\varepsilon,s_{1}w]_{s_{1}\gamma s_{1}}$ in $\mathcal{C}_{s_{1}\gamma s_{1}}$.  Let $\cc'$ be the restriction of the chain in \eqref{eq:chain} to the interval $[s_{1},w]_{\gamma}$, \ie $\cc':s_{1}=x_{1}\lessdot_{\gamma}x_{2}\lessdot_{\gamma}\cdots\lessdot_{\gamma}x_{k}=w$.  The above isomorphism maps $\cc'$ to the maximal chain
	\begin{displaymath}
		\varepsilon=s_{1}x_{1}\lessdot_{\gamma}s_{1}x_{2}\lessdot_{\gamma}\cdots\lessdot_{\gamma}s_{1}x_{k}=s_{1}w,
	\end{displaymath}
	in $[\varepsilon,s_{1}w]_{s_{1}\gamma s_{1}}$, and it sends $y$ to $s_{1}y$ and $z$ to $s_{1}z$.  Since it is an isomorphism, we have $s_{1}y\lessdot_{s_{1}\gamma s_{1}}s_{1}z$, and since $\ell_{S}(s_{1}w)=k-1<k$, we conclude by induction on length that the we have 
	\begin{multline*}
		\text{either}\quad s_{1}x_{j}\wedge_{s_{1}\gamma s_{1}} s_{1}y=s_{1}x_{j}\wedge_{s_{1}\gamma s_{1}} s_{1}z\quad\text{or}\\
			s_{1}x_{j}\vee_{s_{1}\gamma s_{1}} s_{1}y=s_{1}x_{j}\vee_{s_{1}\gamma s_{1}} s_{1}z\quad\text{(but not both)}. 
		\end{multline*}
	(Here the subscripts ``$s_{1}\gamma s_{1}$'' are supposed to indicate that joins and meets are taken within the closed interval $[\varepsilon,s_{1}w]_{s_{1}\gamma s_{1}}$ of $\mathcal{C}_{s_{1}\gamma s_{1}}$.)  Since $[s_{1},w]_{\gamma}$ and $[\varepsilon,s_{1}w]_{s_{1}\gamma s_{1}}$ are isomorphic, we conclude that \eqref{eq:target} holds.	
\end{proof}

An immediate consequence of Lemma~\ref{lem:clm_1} is that $x_{j}$ is left-modular as an element of the interval $[s_{1},w]_{\gamma}$.

\begin{lemma}\label{lem:clm_2}
	If $s_{1}\leq_{\gamma}z\leq_{\gamma}w$, but $s_{1}\not\leq_{\gamma}y$, then \eqref{eq:target} is satisfied.
\end{lemma}
\begin{proof}
	If $s_{1}\not\leq_{\gamma}y$ but $s_{1}\leq_{\gamma}z$, then Lemma~\ref{lem:covering_join} implies $z=s_{1}\vee y$, which in turn implies $x_{j}\vee z=x_{j}\vee (s_{1}\vee y)=x_{j}\vee y$.  Now suppose that $x_{j}\wedge y=x_{j}\wedge z$. Since $s_{1}\leq_{\gamma}x_{j}$ and $s_{1}\leq_{\gamma}z$ it follows that $s_{1}\leq_{\gamma}x_{j}\wedge z=x_{j}\wedge y\leq_{\gamma}y$, which is a contradiction.  Hence \eqref{eq:target} is satisfied.
\end{proof}

\begin{lemma}\label{lem:clm_3}
	If $s_{1}\leq_{\gamma}w$, but $s_{1}\not\leq_{\gamma}z$, then \eqref{eq:target} is satisfied.
\end{lemma}
\begin{proof}
	If $s_{1}\not\leq_{\gamma}z$, then define $y'=s_{1}\vee y$ and $z'=s_{1}\vee z$.  Since $z'$ is an upper bound for both $s_{1}$ and $y$, it follows that $y'=s_{1}\vee y\leq_{\gamma}z'$.  However, equality cannot hold since by Lemma~\ref{lem:covering_join} we have $y\lessdot_{\gamma}z\lessdot_{\gamma}z'$ and $y\lessdot_{\gamma}y'\leq_{\gamma}z'$.  We therefore have $y'<_{\gamma}z'$.  Figure~\ref{fig:illustration} illustrates this situation.  In this figure, straight lines indicate cover relations, and curved lines indicate order relations.  The chain \eqref{eq:chain} is indicated with thick edges, and the interval $[s_{1},w]_{\gamma}$ is indicated by the shaded region.
	
	\begin{figure}
		\centering
		\begin{tikzpicture}\small
			\def\x{1};
			\def\y{1};
			\draw(3*\x,1*\y) node(e){$\varepsilon$};
			\draw(2*\x,2*\y) node(s1){$s_{1}$};
			\draw(1*\x,4.5*\y) node(x){$x_{j}$};
			\draw(4.75*\x,3*\y) node(y){$y$};
			\draw(5.5*\x,4*\y) node(z){$z$};
			\draw(3.75*\x,4*\y) node(yy){$y'$};
			\draw(4.5*\x,5*\y) node(zz){$z'$};
			\draw(3*\x,7*\y) node(w){$w$};
			\draw[thick](e) -- (s1);
			\draw(y) -- (z);
			\draw(y) -- (yy);
			\draw(z) -- (zz);
			\draw(e) .. controls (4*\x,1.5*\y) and (3.75*\x,2.5*\y) .. (y);
			\draw(s1) .. controls (3*\x,2.5*\y) and (2.75*\x,3.5*\y) .. (yy);
			\draw(yy) .. controls (4.25*\x,4.5*\y) and (4*\x,4.5*\y) .. (zz);
			\draw[thick](s1) .. controls (2.25*\x,2.75*\y) and (.75*\x,3.75*\y) .. (x);
			\draw[thick](x) .. controls (2.25*\x,5.25*\y) and (1.75*\x,6.25*\y) .. (w);
			\draw(zz) .. controls (4.75*\x,5.75*\y) and (2.75*\x,6.25*\y) .. (w);
			\begin{pgfonlayer}{background}
				\fill[gray!30!white,rounded corners](1.7*\x,2*\y) -- (2*\x,1.7*\y) -- (2.3*\x,2*\y) -- (5.3*\x,5*\y) -- (3.3*\x,7*\y) -- (3*\x,7.3*\y) -- (2.7*\x,7*\y) -- (-.3*\x,4*\y) -- cycle;
			\end{pgfonlayer}
		\end{tikzpicture}
		\caption{The setup in the proof of Lemma~\ref{lem:clm_3}.}
		\label{fig:illustration}
	\end{figure}
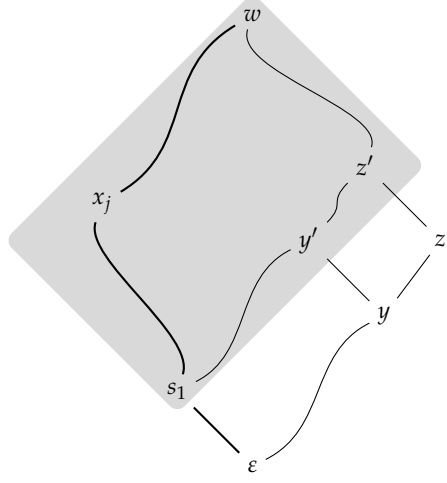
	
	Lemma~\ref{lem:clm_1} implies that $x_{j}$ is left-modular in the interval $[s_{1},w]_{\gamma}$, and since $s_{1}\leq_{\gamma}y'<_{\gamma}z'$ we thus have
	\begin{align}\label{eq:induction_modular}
		(y'\vee x_{j})\wedge z' = y'\vee(x_{j}\wedge z').
	\end{align}	
	Moreover, we have 
	\begin{align}\label{eq:relation}
		x_{j}\vee y' = x_{j}\vee (s_{1}\vee y) = x_{j}\vee y\quad\text{and}\quad x_{j}\vee z' = x_{j}\vee (s_{1}\vee z) = x_{j}\vee z.
	\end{align}	
	
	(i) First suppose that $x_{j}\vee y=x_{j}\vee z$ and $x_{j}\wedge y=x_{j}\wedge z$.  We conclude from \eqref{eq:relation} that $x_{j}\vee y' = x_{j}\vee y = x_{j}\vee z = x_{j}\vee z'$.  Then, the absorption laws in lattices together with \eqref{eq:induction_modular} imply $x_{j}\wedge y'<_{\gamma}x_{j}\wedge z'$.  (Indeed, we clearly have $x_{j}\wedge y'\leq_{\gamma}x_{j}\wedge z'$, and in case of equality we obtain
	\begin{displaymath}
		z' = (z'\vee x_{j})\wedge z' = (y'\vee x_{j})\wedge z' = y'\vee(x_{j}\wedge z') = y'\vee (x_{j}\wedge y') = y' <_{\gamma} z',
	\end{displaymath}
	which is a contradiction.)  We conclude that
	\begin{align*}
		y'\vee (x_{j}\wedge z') & = (y'\vee x_{j})\wedge z' = (z'\vee x_{j})\wedge z' = z',\quad\text{and}\\
		y'\vee z & = s_{1}\vee y\vee z = y\vee z' = z',
	\end{align*}
	but
	\begin{displaymath}
		y'\vee (x_{j}\wedge z'\wedge z) = y'\vee (x_{j}\wedge z) = y'\vee (x_{j}\wedge y) \leq_{\gamma} y'\vee(x_{j}\wedge y') = y' <_{\gamma} z',
	\end{displaymath}
	which contradicts Proposition~\ref{prop:cambrian_semidistributive} stating that $[\varepsilon,w]_{\gamma}$ is semidistributive.  It follows that we cannot have both $x_{j}\vee y=x_{j}\vee z$ and $x_{j}\wedge y=x_{j}\wedge z$ at the same time.
	
	(ii) Now suppose that $x_{j}\vee y<_{\gamma}x_{j}\vee z$ and $x_{j}\wedge y<_{\gamma}x_{j}\wedge z$.  We conclude from \eqref{eq:relation} that $x_{j}\vee y' = x_{j}\vee y <_{\gamma} x_{j}\vee z = x_{j}\vee z'$.  This implies in particular that $z\not\leq_{\gamma}y'\vee x_{j}$.  (Otherwise, we would have $x_{j}\vee z\leq_{\gamma}x_{j}\vee y'<_{\gamma}x_{j}\vee z$, which is a contradiction.)  Since $y\lessdot_{\gamma}z$ and $y\lessdot_{\gamma}y'\leq_{\gamma}y'\vee x_{j}$, we therefore obtain
	\begin{displaymath}
		y\leq_{\gamma}(y'\vee x_{j})\wedge z<_{\gamma}z,
	\end{displaymath}
	which in view of $y\lessdot_{\gamma}z$ implies $y=(y'\vee x_{j})\wedge z$.  Now, since $x_{j}\leq_{\gamma}y'\vee x_{j}$ we have in particular 
	\begin{displaymath}
		x_{j}\wedge z\leq_{\gamma}(y'\vee x_{j})\wedge z = y,
	\end{displaymath}	
	which yields the contradiction $x_{j}\wedge z\leq_{\gamma}x_{j}\wedge y<_{\gamma}x_{j}\wedge z$.

	We have thus shown that \eqref{eq:target} is satisfied.
\end{proof}

\begin{lemma}\label{lem:clm_4}
	If $s_{1}\not\leq_{\gamma}w$, then \eqref{eq:target} is satisfied.
\end{lemma}
\begin{proof}
	Proposition~\ref{prop:sortable_recursion} implies $[\varepsilon,w]_{\gamma}\cong[\varepsilon,w]_{s_{1}\gamma}$, and the result follows by induction on rank, since we can view $[\varepsilon,w]_{s_{1}\gamma}$ as an interval in the $s_{1}\gamma$-Cambrian semilattice of the standard parabolic subgroup $W_{\langle s_{1}\rangle}$ of rank $n-1$.  
\end{proof}

In the lattice in Figure~\ref{fig:cambrian_c3_interval}, the chain \eqref{eq:chain} is indicated by thick edges, and the left-modularity of its elements can be verified by hand.

Next we want to show that every closed interval in $\camb$ is extremal.  To reduce the amount of work, we recall more helpful results.  First we observe that in a semidistributive lattice the number of join-irreducible elements always equals the number of meet-irreducible elements.

\begin{lemma}[\cite{day79characterizations}*{Lemma~3.1}]\label{lem:semidistributive_same_irreducibles}
	If $\LL$ is semidistributive, then $\bigl\lvert\mathcal{J}(\LL)\bigr\rvert=\bigl\lvert\mathcal{M}(\LL)\bigr\rvert$.
\end{lemma}

The next result states that the set of elements in a Coxeter group that do not lie above some simple generator is closed under joins whenever they exist.

\begin{lemma}[\cite{reading11sortable}*{Corollary~2.21}]\label{lem:join_parabolic}
	Let $(W,S)$ be a Coxeter system, let $u,v\in W$ and $s\in S$.  If $s\not\leq_{S}u$ and $s\not\leq_{S}v$ as well as $u\vee v$ exists, then $s\not\leq_{S}u\vee v$.  
\end{lemma}

Since joins (whenever they exist) and meets in $\camb$ agree with those in the weak order lattice on $W$, a statement analogous to Lemma~\ref{lem:join_parabolic} holds for $\camb$.  Before we enumerate the meet-irreducible elements in some interval $[\varepsilon,w]_{\gamma}$, we prove another technical lemma.

\begin{lemma}\label{lem:meet_irreducibles}
	Let $w\in C_{\gamma}$ with $s_{1}\leq_{\gamma}w$, and let $u,v\leq_{\gamma}w$ such that $s_{1}\not\leq_{\gamma}u$ and $s_{1}\not\leq_{\gamma}v$.  If $v$ is meet-irreducible, then $u\leq_{\gamma}v$.  
\end{lemma}
\begin{proof}
	Since $s_{1}\not\leq_{\gamma}u$ and $s_{1}\not\leq_{\gamma}v$, Lemma~\ref{lem:join_parabolic} implies $s_{1}\not\leq_{\gamma}u\vee v$.  It follows that $s_{1}\vee v\not\leq_{\gamma}u\vee v$.  Lemma~\ref{lem:covering_join} implies that $v\lessdot_{\gamma}s_{1}\vee v$, and we trivially have $v\leq_{\gamma}u\vee v$.  If $v$ is meet-irreducible, then $v$ has exactly one upper cover, which forces $u\vee v=v$, and therefore $u\leq_{\gamma}v$.  
\end{proof}

\begin{proposition}\label{prop:cambrian_extremal}
	For $w\in C_{\gamma}$ we have $\bigl\lvert\mathcal{M}\bigl([\varepsilon,w]_{\gamma}\bigr)\bigr\rvert=\ell_{S}(w)$.
\end{proposition}
\begin{proof}
	We proceed again by induction on rank and length.  If $W$ has rank $2$ or if $\ell_{S}(w)\leq 2$, then the result is trivially true.  Hence let $W$ have rank $n$, let $\ell_{S}(w)=k$, and suppose that the claim is true for all standard parabolic subgroups of $W$ of rank $<n$, and for all $\gamma'$-sortable elements $w'\in W$ for some Coxeter element $\gamma'\in W$ with $\ell_{S}(w')<k$.  We distinguish two cases:
	
	\smallskip	
	
	(i) Assume that $s_{1}\leq_{\gamma}w$.  Proposition~\ref{prop:sortable_recursion} implies that the interval $[s_{1},w]_{\gamma}$ is isomorphic to the interval $[\varepsilon,s_{1}w]_{s_{1}\gamma s_{1}}$, and it follows by induction on length that there are $k-1$ meet-irreducible elements in $[s_{1},w]_{\gamma}$, and it is straightforward to show that $\mathcal{M}\bigl([s_{1},w]_{\gamma}\bigr)\subseteq\mathcal{M}\bigl([\varepsilon,w]_{\gamma}\bigr)$.
	
	Now we show that there is exactly one additional meet-irreducible element in $[\varepsilon,w]_{\gamma}$.  Consider the set $C=\{u\leq_{\gamma}w\mid s_{1}\not\leq_{\gamma}u\}$.  Lemma~\ref{lem:meet_irreducibles} implies that any element in $\mathcal{M}\bigl([\varepsilon,w]_{\gamma}\bigr)\setminus\mathcal{M}\bigl([s_{1},w]_{\gamma}\bigr)$ is maximal in $C$.  Since $w$ is an upper bound for $C$, Theorem~\ref{thm:sortable_meets_joins} implies that $x=\bigvee C$ exists.  Lemma~\ref{lem:join_parabolic} implies that $x\in C$, and it is thus the unique maximal element in $C$.  Hence $\mathcal{M}\bigl([\varepsilon,w]_{\gamma}\bigr)\setminus\mathcal{M}\bigl([s_{1},w]_{\gamma}\bigr)=\{x\}$, and we obtain $\bigl\lvert\mathcal{M}\bigl([\varepsilon,w]_{\gamma}\bigr)\bigr\rvert=k$ as desired.
	
	\smallskip	
	
	(ii) Assume that $s_{1}\not\leq_{\gamma}w$.  Proposition~\ref{prop:sortable_recursion} implies $[\varepsilon,w]_{\gamma}\cong[\varepsilon,w]_{s_{1}\gamma}$, and the result follows by induction on rank, since we can view $[\varepsilon,w]_{s_{1}\gamma}$ as an interval in the $s_{1}\gamma$-Cambrian semilattice of the standard parabolic subgroup $W_{\langle s_{1}\rangle}$ of rank $n-1$. 
\end{proof}

\begin{example}\label{ex:cambrian_3}
	Let us continue Example~\ref{ex:cambrian_2}, in which we in particular listed the meet-irreducible elements in the interval of the $s_{0}s_{1}s_{2}s_{3}$-Cambrian semilattice of $\tilde{C}_{3}$ shown in Figure~\ref{fig:cambrian_c3_interval}.  Among these, only $s_{2}s_{3}s_{2}s_{3}$ does not lie above $s_{0}$.
\end{example}

We conclude this article with the proof of Theorem~\ref{thm:cambrian_trim}. 

\begin{proof}[Proof of Theorem~\ref{thm:cambrian_trim}]
	First, let $w\in C_{\gamma}$ with $\ell_{S}(w)=k$.  Proposition~\ref{prop:cambrian_left_modular} implies that $[\varepsilon,w]_{\gamma}$ is left-modular, and since Proposition~\ref{prop:cambrian_semidistributive} implies that $[\varepsilon,w]_{\gamma}$ is semidistributive, it follows from Lemma~\ref{lem:semidistributive_same_irreducibles} and Proposition~\ref{prop:cambrian_extremal} that
	\begin{displaymath}
		\mathcal{J}\bigl([\varepsilon,w]_{\gamma}\bigr)=\mathcal{M}\bigl([\varepsilon,w]_{\gamma}\bigr)=k.
	\end{displaymath}
	Hence $[\varepsilon,w]_{\gamma}$ is extremal, which by definition implies that $[\varepsilon,w]_{\gamma}$ is trim. 
	
	Now let $u,v\in C_{\gamma}$ with $u\leq_{\gamma}v$.  The interval $[u,v]_{\gamma}$ is clearly a subinterval of $[\varepsilon,v]_{\gamma}$.  The first part of this proof shows that $[\varepsilon,v]_{\gamma}$ is trim, and Theorem~\ref{thm:trim_intervals} implies the same for $[u,v]_{\gamma}$.  
	
	The distributivity of graded closed intervals in $\camb$ follows now from Theorem~\ref{thm:graded_extremal_distributive}.
\end{proof}

\section*{Acknowledgements}
	\label{sec:acknowledgements}
I thank the anonymous referee for the careful reading, and for many helpful suggestions on how to improve the presentation of the paper.  I am in particular grateful for a suggestion on how to simplify the proof of Proposition~\ref{prop:cambrian_extremal}.  

\bibliography{../../literature}

\end{document}